\documentclass[graybox]{svmult}

\usepackage{mathptmx}       
\usepackage{helvet}         
\usepackage{courier}        
\usepackage{type1cm}        
\usepackage{makeidx}         
\usepackage{graphicx}        
\usepackage{multicol}        

\usepackage{amssymb,amsfonts,amscd,amsmath,amscd}

\input amssym.def

\input amssym.tex

\DeclareFontFamily{OT1}{pzc}{}
\DeclareFontShape{OT1}{pzc}{m}{it}%
             {<-> s * [0.900] pzcmi7t}{}
\DeclareMathAlphabet{\mathscr}{OT1}{pzc}%
                                 {m}{it}

\newcommand\be{\begin{equation}}
\newcommand\ee{\end{equation}}
\newcommand\bea{\begin{eqnarray}}
\newcommand\eea{\end{eqnarray}}
\newcommand\bi{\begin{itemize}}
\newcommand\ei{\end{itemize}}
\newcommand\ben{\begin{enumerate}}
\newcommand\een{\end{enumerate}}

\newtheorem{thm}{Theorem}[section]

\newtheorem{lem}[thm]{Lemma}

\newtheorem{prop}[thm]{Proposition}

\newtheorem{rek}[thm]{Remark}
\numberwithin{subsubsection}{subsection}

\begin{document}

\title*{Most Subsets are Balanced in Finite Groups}

\author{Steven J. Miller and Kevin Vissuet}
\institute{Department of Mathematics and Statistics, Williams College, sjm1@williams.edu (Steven.Miller.MC.96@aya.yale.edu) and Department of Mathematics, UCSD, kvissuet@ucsd.edu. The first named author was partially supported by NSF Grant DMS0970067, and the second named author was partially supported by NSF Grant DMS0850577. We thank the participants of the 2012 SMALL REU program, especially Ginny Hogan and Nicholas Triantafillou, as well as Kevin O'Bryant, for helpful discussions.}

%
%
\maketitle

\abstract{The sumset is one of the most basic and central objects in additive number theory. Many of the most important problems (such as Goldbach's conjecture and Fermat's Last theorem) can be formulated in terms of the sumset $S + S = \{x+y : x,y\in S\}$ of a set of integers $S$. A finite set of integers $A$ is sum-dominant if $|A+A| > |A-A|$. Though it was believed that the percentage of subsets of $\{0,\dots,n\}$ that are sum-dominant tends to zero, in 2006 Martin and O'Bryant proved a very small positive percentage are sum-dominant if the sets are chosen uniformly at random (through work of Zhao we know this percentage is approximately $4.5 \cdot 10^{-4}$). While most sets are difference-dominant in the integer case, this is not the case when we take subsets of many finite groups. We show that if we take subsets of larger and larger finite groups uniformly at random, then not only does the probability of a set being sum-dominant tend to zero but the probability that $|A+A|=|A-A|$ tends to one, and hence a typical set is balanced in this case. The cause of this marked difference in behavior is that subsets of $\{0, \dots, n\}$ have a fringe, whereas finite groups do not. We end with a detailed analysis of dihedral groups, where the results are in striking contrast to what occurs for subsets of integers.
}




\section{Introduction}

Given a subset $S$ of a group $G$, we define its sumset $S+S$ and difference set $S-S$ by \bea S + S & \ = \ & \{a_i + a_j: a_i, a_j \in A\} \nonumber\\ S - S & \ = \  & \{a_i - a_j: a_i, a_j \in A\}, \eea and let $|X|$ denote the cardinality of $X$. Notice that we're writing the group action as addition, but are not assuming commutativity. If we were to write the action multiplicatively we would still call these the sumset and the difference set, instead of the product and quotient sets, to match the language from earlier work which studied subsets of the integers.

If $|S+S| > |S-S|$ then $S$ is sum-dominant or an MSTD (more sums than differences) set, while if $|S+S| = |S-S|$ we say $S$ is balanced and if $|S+S| < |S-S|$ then $S$ is difference-dominant. If we let the group $G$ be the integers, then we expect that for a `generic' set $S$ we have $|S-S| > |S+S|$. This is because addition is commutative while subtraction is not, since a typical pair $(x,y)$ contributes one sum and two differences.

Though MSTD sets are rare among all finite subsets of integers, they do exist. Examples of MSTD sets go back to the 1960s. Conway is credited with finding $\{0, 2, 3, 4$, $7$, $11$, $12$, $14\}$; for other early examples see also Marica \cite{Ma} and Freiman and Pigarev \cite{FP}. Recently there has been much progress in finding infinite families, either through explicit constructions (see Hegarty \cite{He} and Nathanson \cite{Na1}), and existance arguments via non-constructive methods (see Ruzsa \cite{Ru1, Ru2, Ru3} and Miller-Orosz-Scheinerman \cite{MOS}). The main result in the subject is due to Martin and O'Bryant \cite{MO}, who proved a positive percentage of subsets of $\{0, 1, \dots, N\}$ are sum-dominant, though the percentage is small (work of Zhao \cite{Zh2} suggests it is around $4.5 \cdot 10^{-4}$).

Almost all previous research on MSTD sets focused exclusively on subsets of the integers, though recently Zhao \cite{Zh1} extended previous results of Nathanson \cite{Na2}, who showed that MSTD sets of integers can be constructed from MSTD sets in finite abelian groups. Zhao provides asymptotics for the number of MSTD sets in finite abelian groups. An immediate corollary of the main theorem in \cite{Zh2} is that if $\{G_n\}$ is a sequence of finite abelian groups with $\{G_n\}\rightarrow \infty$ then the percentage of MSTD sets is almost surely $0$. In this paper we not only extend this result to difference-dominant sets but to non-abelian finite groups as well.

\begin{thm}\label{thm:unifchosenrandG1}
Let $\{G_n\}$ be a sequence of finite groups, not necessarily abelian, with $|G_n|\rightarrow \infty$. Let $S_n$ be a uniformly chosen random subset of a $G_n$. Then $\mathbb{P}(S_n+S_n=S_n-S_n=G)\rightarrow 1$ as $n\to\infty$. In other words, as the size of the finite group grows almost all subsets are balanced (with sumset and difference set the entire group).
\end{thm}

While Theorem \ref{thm:unifchosenrandG1} shows that in the limit almost all subsets of finite groups are balanced, it leaves open the relative behavior of sum-dominant and difference-dominant sets. Though the numbers of such sets are lower order and percentagewise tend to zero, are there more, equal or fewer sum-dominant or difference-dominant sets? For example, Figure \ref{fig:clockgroupinvestigation} shows the result of numerical simulations for 10,000 clock groups $\mathbb{Z}/n\mathbb{Z}$ for $n \in \{10, \dots, 100\}$.

\begin{figure}
\begin{center}
\scalebox{.996525}{\includegraphics{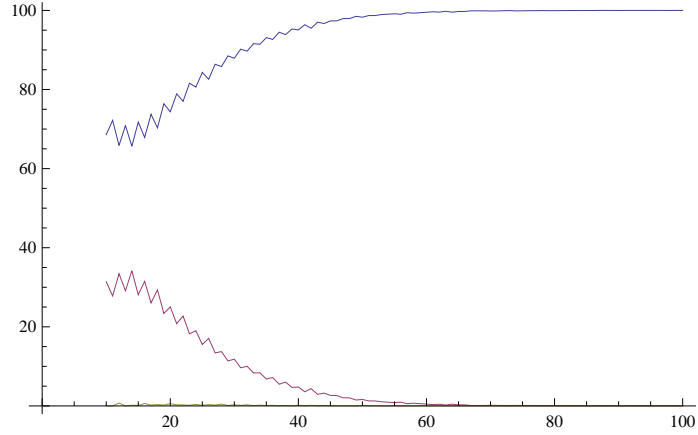}}
\caption{\label{fig:clockgroupinvestigation} Numerical simulations on the number of balanced, difference-dominant and sum-dominant subsets of $\mathbb{Z}/n\mathbb{Z}$ for $n \in \{10, \dots, 100\}$. For each $n$ we uniformly chose 10,000 random subsets of $\{1, \dots, n\}$. Top plot is the percentage of balanced, middle is the percentage of difference-dominant, and bottom is the percentage of sum-dominant.}
\end{center}\end{figure}

In Section \ref{sec:dih} we explore this question for subsets of dihedral groups, and see very different behavior than in the integers. We conjecture that while almost all subsets of the dihedral group are balanced, there are more MSTD sets than there are difference-dominant sets, in sharp contrast to the prevalence of difference-dominant subsets of the integers.

The paper is organized as follows. We first prove our main result for all finite groups in \S\ref{sec:subsetsfinitegroups}. We then explore the MSTD sets of the dihedral group in \S\ref{sec:dihedral}. We end with some concluding remarks and suggestions for future research.

\section{Subsets of Finite Groups}\label{sec:subsetsfinitegroups}

Martin and O'Bryant \cite{MO} showed that although MSTD subsets of the integers are rare, they are a positive percentage of subsets. MSTD sets in finite groups are even rarer. We will prove that as the size of a finite group tends to infinity, the probability that a subset chosen uniformly at random is sum-dominant tends to zero. Somewhat surprisingly, this is also true for difference-dominant sets. This is very different than the integer case, where more than 99.99\% of all subsets are difference-dominant.

The reason the integers behave differently than finite groups is that a subset of the integers contains fringe elements, which we now define. Let $S$ be a subset of $I_n := \{0, 1, \dots, n\}$ chosen uniformly at random. The elements of $S$ near 0 and $n$ are called the fringe elements. Interestingly the notion of nearness is independent of $n$; the reason is that almost all possible elements of $I_n + I_n$ and $I_n - I_n$ are realized respectively by $S+S$ and $S-S$; Martin and O'Bryant \cite{MO} prove that $S+S$ and $S-S$ miss on average 10 and 6 elements, while Lazarev, Miller and O'Bryant \cite{LMO} prove the variance is bounded independent of $n$. Thus whether or not a set is sum-dominant is essentially controlled by the fringe elements of $S$, as the `middle' is filled with probability 1 and the presence and absence of fringe elements control the extremes. In a finite group, there are no fringe elements since each element can be written as $|G|$ different sums and differences, and thus most elements appear in the sumset or difference set with high probability.

In the proof of Theorem \ref{thm:unifchosenrandG1} we reduce certain probabilities to products of Lucas numbers $L(n)$; these satisfy the recurrence $L(n+2) = L(n+1) + L(n)$ with initial conditions $L(0) = 2$ and $L(1) = 1$. Note this is the same recurrence relation as the Fibonacci numbers $F(n)$, who differ from the Lucas numbers in that their initial conditions are $F(0) = 0$ and $F(1) = 1$.

The following lemma is useful, and is in the spirit of calculations from \cite{LMO}. The interpretation will be that the red vertices correspond to elements chosen to be in an $S$, and the condition that no neighboring vertices are both colored red will ensure that certain elements are not represented in $S+S$.

\begin{lem}\label{lem:numberwayscolorngon} Let $C_n= \{a_1, \dots , a_n\}$ denote a closed chain of $n$ elements (so $a_1$ is adjacent to $a_2$ and $a_n$, and so on). If $P(n)$ is the number of ways to color the vertices of $C_n$ red or blue such that no two neighboring vertices are colored red, then $P(n) = L(n)$.
\end{lem}

\begin{proof} We derive a recurrence formula for $P(n)$. We may draw $C_n$ as a regular $n$-gon with the $a_i$'s as the vertices. Let $A(n)$ denote the number of ways a line with $n$ vertices $a_1, a_2, \dots, a_n$ can be colored red or blue so that no two neighboring vertices are colored red. We have \be P(n) \ = \ A(n-1) + A(n-3). \ee

To see this, there are two cases. Consider the first vertex, $a_1$. If it is colored blue then we may `break' the chain at $a_1$ and the problem reduces to determining the number of ways to color $n-1$ vertices on a line red or blue so that no two neighboring ones are both red; by definition this is $A(n-1)$. Alternatively, if $a_1$ is colored red then $a_2$ and $a_n$ must both be colored blue, and thus we are left with coloring $n-3$ vertices on a line so that no two consecutive vertices are both red; again, by definition this is just $A(n-3)$.

%

Thus the lemma is reduced to computing $A(n)$, which satisfies the Fibonacci-Lucas recurrence. To see this, consider $n$ vertices on a line, with $A(n)$ the number of ways to color these red and blue so that no neighbors are both colored red. If the first vertex is colored blue, then by definition there are $A(n-1)$ ways to color the remaining vertices, while if the first vertex is colored red then the second must be colored blue, leaving $A(n-2)$ ways to color the remaining vertices. Thus \be A(n)  \ = \ A(n-1) + A(n-2). \ee It is easy to see that $A(1)=2$ and $A(2)=3$, which implies \be A(n)\ =\ F(n+2), \ee where $F(n)$ is the $n$\textsuperscript{th} Fibonacci number. As $P(n) = A(n-1) + A(n-3)$, we find \be P(n) \ = \ F(n+1)+F(n-1). \ee As the $n$\textsuperscript{th} Lucas number satisfies \be L(n) \ = \ F(n+1) + F(n-1) \ee (this can easily be proved directly, or see for example \cite{BQ}), we find $P(n) = L(n)$ as claimed. \hfill $\Box$ \end{proof}

We now prove our main theorem.\\ \

\noindent \emph{Proof of Theorem \ref{thm:unifchosenrandG1}:} We start by showing that the probability a $g \in G = \{g_1, g_2$, $\dots$, $g_n\}$ is in $S+S$ approaches 1 exponentially fast. For $g\in G$, we have
\begin{equation}
\mathbb{P}(g\notin S+S)\ = \ \mathbb{P}(x\notin S \vee y \notin S \quad \forall x,y \in G  \text{ s.t. }  x+y=g).
\end{equation}
To determine the probability that $S+S$ is not all of $G$ we will add the probabilities $\mathbb{P}(g\notin S+S)$ for each $g$. Note these probabilities are not independent, as $x \not\in G$ affects the probability of several $g$ being in $S+S$.

We concentrate on a fixed $g$. If $x\in G$ then there exist a chain of elements $\{x_1,x_2\dots x_n\}$ $=$ $X\subseteq G$ such that
\be\label{eq:keyeqgroups} x+x_1 \ = \ x_2+x_3\ = \ \cdots\ =\ x_{n-1}+ x_n \ = \ x_n+x \ = \ g, \ee; clearly the pairs depend on $g$. Note that $X$ also depends on the choice of $x\in G$. If we denote all distinct chains as $X_1,\dots,X_n$ then these sets partition $G$. If $S$ is a subset of $G$, for $g$ not to be represented in $S+S$ we need at least one element of each pair in each $X_i$ to fail to be in $S$. The number of ways this can happen is $\prod L(|X_i|)$, where $L(n)$ is the $n$\textsuperscript{th} Lucas number.


To see this equality we use a method similar to that used by Lazarev, Miller, and O'Bryant in \cite{LMO}. Counting the number of subsets of $X_i$ such that we never take two adjacent elements is equivalent to counting the number of ways the vertices of a regular polygon with $|X_i|=n$ vertices can be colored with two colors (say red and blue) such that no two adjacent vertices are blue. Note that each subset $S$ of vertices with this property is equivalent to a set where $g \not\in S+S$, and since the $X_i$ partition $G$, then by Lemma \ref{lem:numberwayscolorngon} the number of such colorings is $\prod L(|X_i|)$.
Combining the independence of the $X_i$ with Lemma \ref{lem:numberwayscolorngon}, we conclude,
 \begin{equation}
\mathbb{P}(g\notin S+S)\ = \ \frac{\prod L(|X_i|)}{2^{|G|}}.
\end{equation}
For example, take the element $a+b \in D_{6}=\langle a,b | a+a+a , b+b, a+b+a+b\rangle$, where $D_6$ is the dihedral group with six elements. Here we have that \begin{equation}a+b=(a+b)+(a+a+a)=(a+a+a)+(a+b)\end{equation} and \begin{equation}a+b=(a+a)+(a+a+b)=(a+a+b)+(a)=(a)+(b)=(b)+(a+a),\end{equation} where plus denotes the group operation. The two chains we obtain are $X_1=\{a+b, a+a+a\}$ and $X_2=\{a+a, a+a+b, a, b\}$. Letting $S_{X_1}=S\cap X_1$ and $S_{X_2}=S\cap X_2$ we have that
\bea \mathbb{P}(a+b\notin S+S) & \ = \  & \mathbb{P}(a+b\notin S_{X_1}+S_{X_1})\mathbb{P}(a+b\notin S_{X_2}+ S_{X_2} ) \nonumber\\ &=& \left(\frac{L(2)}{2^2}\right)\left(\frac{L(4)}{2^4}\right), \eea where the latter equality occurs because of Lemma \ref{lem:numberwayscolorngon}.

Note that $L(n)= \phi^n+(-\phi)^{-n}$ where $\phi = \frac{1+\sqrt{5}}2$ is the golden ratio. As the $X_i$'s are disjoint, we obtain for each $g\in G$ that
\begin{equation}
\mathbb{P}(g\notin S+S)\ = \ \frac{\prod L(|X_i|)}{2^{|G|}}\ \le \  \frac{\prod 1.8^{|X_i|}}{2^{|G|}}\ = \ \left(\frac{1.8}{2}\right)^{|G|}.
\end{equation}  As crude bounds suffice, we use the union bound to bound the contribution from each element in $G$, and find
\begin{equation}
\mathbb{P}(|S+S|<|G|)\ = \ \mathbb{P}(\cup_{g\in G} g\notin G) \ \le \  \sum_{g\in S+S} \mathbb{P}(g\notin S+S)\ \le \  |G|\left(\frac{1.8}{2}\right)^{|G|}.
\end{equation}
As the size of the group approaches infinity, $\mathbb{P}(|S+S|<|G|)$ approaches zero. The same argument holds for $S-S$ since there is a one to one bijection between group elements and their inverses. Thus most subsets are balanced. \hfill $\Box$ \ \\

\begin{rek} The above arguments do not apply to subsets of the integers. The reason is due to the lack of a group structure. In particular, the result from equation \eqref{eq:keyeqgroups} does not hold and different elements have different numbers of representations as a sum or a difference. For example, for the integers the number of pairs $(x,y) \subset \{0, \dots, n-1\}^2$ such that $x+y=k$ is a triangular function of $k$, peaking when $k=n-1$. Thus whether or not small (near 0) or large (near $2n-2$) $k$ are in the sumset is controlled by the fringe elements of our set. A similar result holds for differences, and thus if the fringe is carefully chosen then we can force our set to be sum-dominant or difference-dominant. Note such forcing arguments cannot happen with a group structure.
\end{rek}


Note that we used $1.8$ as a very crude bound. While $\prod L(|X_i|)$ is much closer to $\phi^n$ then it is to $1.8^n$, since $\phi^0$ is less than $L(0)$, $\phi^n$ does not provide an inequality for all $n$.


\section{Sum Dominated sets in Dihedral Groups}\label{sec:dihedral}

Although sum-dominant sets and and difference-dominant sets are rare in sufficiently large finite groups, we can compare the size of the number of sum-dominant subsets and difference-dominant subsets in any fixed finite group. In this section we first explore the sumset and difference set of cyclic groups. We then apply those results to give intuition on why in any dihedral group, there should be more sum-dominant sets than difference-dominant sets.

\subsection{Cyclic Group Preliminaries}

Before we look at the dihedral group, we explore two different cases in cyclic groups. In the first case we compute the probability of an element missing in the sumset and difference set. In the second case we compute the probability of missing an element in $A+B$ where $A$ and $B$ are both subsets of $\mathbb{Z}/n\mathbb{Z}$.

\begin{lem}\label{lem:spluss}
Let $S$ be a uniformly chosen random subset of $\mathbb{Z}/n\mathbb{Z}$. Then
\begin{equation}
\mathbb{P}(k \notin S + S) \ = \  O\left((3/4)^{n/2}\right).
\end{equation}
\end{lem}

\begin{proof}
Let $k\in \mathbb{Z}/n\mathbb{Z}$. Since addition is commutative, all sets of pairs of elements that sum to $k$ partition the group. Furthermore, the number of pairs of distinct elements in $\mathbb{Z}/n\mathbb{Z}$ is equal to either $n/2,n/2-1$ or $(n-1)/2$. The number of distinct pairs depends on the parity of $n$ and $k$. From the independence of the pairs of elements that sum to $k$, we have
\begin{equation}
\mathbb{P}(k \notin S+S)\ = \ \prod_{0\leq i\leq \lceil (n+1)/2 \rceil}\mathbb{P}({i \notin S \vee k-i \notin S}).
\end{equation}
Finally, since counting the number of distinct pairs is straightforward, we conclude
\begin{equation}
\mathbb{P}(k \notin S+S) \ = \  \left\{
\begin{array}{c l}
(1/2)^2(3/4)^{n/2-1}& \mbox{ $k$ even and $n$ even } \\
(3/4)^{n/2}& \mbox{ $k$ odd and $n$ even } \\
(1/2)(3/4)^{(n-1)/2}& \mbox{ $n$ odd. }
 \end{array}
\right.
\end{equation}
The factor of $1/2$ is due to the number of elements $x\in \mathbb{Z}/n\mathbb{Z}$ such that $x+x=k$. Again, the number of these elements depends on the parity of $n$ and $k$. \hfill $\Box$
\end{proof}

\begin{lem}\label{lem:s1pluss2}
Let $S_1$ and $S_2$ be independent sets chosen uniformly among subsets of $\mathbb{Z}/n\mathbb{Z}$. Then
\begin{equation}
\mathbb{P}(k \notin S_1 + S_2) \ = \  (3/4)^n.
\end{equation}
\end{lem}

\begin{proof}
Let $k\in \mathbb{Z}/n\mathbb{Z}$. The claim follows immediately from the fact that
\begin{equation}
\mathbb{P}(k \notin S_1+S_2)\ = \ \prod_{0\leq i\leq n-1}\mathbb{P}({i \notin S_1 \vee k-i \notin S_2})
\end{equation} and the fact that these $n$ products are mutually independent. \hfill $\Box$
\end{proof}

\begin{lem} \label{lem:SminusS}
Let $S$ be a uniformly chosen random subset of $\mathbb{Z}/n\mathbb{Z}$. Then \begin{equation}  \mathbb{P}( k\notin S - S )\ = \ \frac{L(n/d)^d}{2^n}\ = \  O\left((\phi/2)^n\right),  \end{equation} where $\gcd(k,n)=d$, $L(n)$ is the $n$\textsuperscript{{\rm th}} Lucas number, and $\phi$ is the golden ratio.
\end{lem}

\begin{proof}
Let $k\in \mathbb{Z}/n\mathbb{Z}$. Since the order of $k$ in $\mathbb{Z}/n\mathbb{Z}$ is equal to $n/\gcd(n,k)$, if we have a set $\{x_1,x_2,\dots, x_m\}$ such that $x_1-x_2=x_2-x_3=\dots=x_m-x_1=k$ then $m=n/\gcd(n,k)$. These sets partition the group and thus, the number of subsets of $\mathbb{Z}/n\mathbb{Z}$ that satisfy this property is $\gcd(n,k)$. Combining the fact that these sets have a pairwise trivial intersection with Lemma \ref{lem:numberwayscolorngon} we have
\begin{equation}  \mathbb{P}(k\notin S - S )\ = \ \frac{L(n/d)^d}{2^n}, \end{equation}
as desired. \hfill $\Box$
\end{proof}

\begin{lem}\label{lem:S1minusS2} Let $S_1$ and $S_2$ be independent sets chosen uniformly among subsets of $\mathbb{Z}/n\mathbb{Z}$. Then
\begin{equation}
 \mathbb{P}( k\notin S_1 - S_2 )\ = \ \left(\frac{3}{4}\right)^n.
 \end{equation}
\end{lem}

\begin{proof} The proof follows immediately from the following equalities:
\begin{eqnarray}
\mathbb{P}(k\notin S_1-S_2)&\ = \ & \prod_{x \in \mathbb{Z}/n\mathbb{Z}}\mathbb{P}(x\notin S_1 \cup x-k \notin S_2)\nonumber\\
&\ = \ & \prod_{x\in \mathbb{Z}/n\mathbb{Z}}(1-\mathbb{P}(x\in S_1 \cap x-k \notin S_2))\nonumber\\
&\ = \ & \prod_{x\in \mathbb{Z}/n\mathbb{Z}}(1-\mathbb{P}(x\in S_1)\mathbb{P}( x-k \notin S_2))\nonumber\\
&\ = \ & \left(\frac{3}{4}\right)^n.
\end{eqnarray} \ \\ \hfill $\Box$
\end{proof}

\begin{prop} Let $S$ be uniformly chosen random subsets of $\mathbb{Z}/n\mathbb{Z}$ then as $n$ approaches infinity  $\mathbb{P}(|S+S|=|S-S|=n)$ approaches 1.
\end{prop}

\begin{proof}
This is immediate from the union bound and Lemmas \ref{lem:spluss} through \ref{lem:S1minusS2}. \hfill $\Box$
\end{proof}

\subsection{Dihedral Group Case}\label{sec:dih}

Let $S$ be a subset of $D_{2n}=\langle a,b | a^n , b^2, abab\rangle$ chosen uniformly at random. We first give a proof for the dihedral group subcase of Theorem \ref{thm:unifchosenrandG1} by using the previous lemmas. Before we do so we need two results. The first looks at the probability of a rotation element ($k=a^i$) not being in the sumset. The second looks at the probability of a reflection element $(k=a^ib)$ not being in the sumset. We denote the set of all rotation elements by $R$ and the set of all reflection elements by $F$.

\begin{lem}\label{lem:rotate}
Let $S$ be a uniformly chosen random subset of $D_{2n}$ and let $k\in D_{2n}$ such that $k=a^i$. Then $\mathbb{P}(k\notin S+S)\leq (3/4)^{n/2}(\phi/2)^n$ and $\mathbb{P}(k\notin S-S)\leq (\phi/2)^{2n}$.
\end{lem}

\begin{proof}
An element of the form $a^i$ can be written as a product of two rotations,  $a^xa^y$ where $x+y=i$, or the product of two reflections, $a^xba^yb$ where $x-y=i$.  Since the set of rotations and the set of reflections can be viewed as cyclic groups the proofs follow immediately from Lemmas  \ref{lem:spluss} and  \ref{lem:SminusS}. \hfill $\Box$
\end{proof}

\begin{lem}\label{lem:flip}
Let $S$ be a uniformly chosen random subset of $D_{2n}$ and let $k\in D_{2n}$ such that $k=a^ib$. Then $\mathbb{P}(k\notin S+S)\leq (3/4)^n$ and $\mathbb{P}(k\notin S-S)\leq (3/4)^n$.
\end{lem}

\begin{proof}
Since an element of the form $a^ib$  can be written as a product of a rotation and a reflection the proof follows immediately from Lemma \ref{lem:s1pluss2}. \hfill $\Box$
\end{proof}

\begin{thm}\label{thm:dih}
Let $S$ be a uniformly random subset of $D_{2n}$. Then, as $n$ approaches infinity, $\mathbb{P}(|S+S|=|S-S|)$ approaches 1.
\end{thm}

\begin{proof}
The proof follows immediately from applying the union bound to Lemmas \ref{lem:rotate} and \ref{lem:flip}. \hfill $\Box$
\end{proof}

Note that by Theorem \ref{thm:unifchosenrandG1} we know that the percentage of sum-dominant and difference-dominant sets goes to zero at an exponential rate. However, if we look at any fixed $D_{2n}$ we conjecture that the number of sum-dominant subsets is greater than the number of difference-dominant subsets. For the first few dihedral groups (up to $D_{16}$) Figure \ref{fig:charts} shows an exhaustive comparison of the subsets of $D_{2n}$. Figure \ref{fig:charts} also includes a sample statistic for larger dihedral groups. Note that it is hard to continue a complete enumeration.

\begin{figure}[h]
\begin{center}
\includegraphics[width=5.5cm]{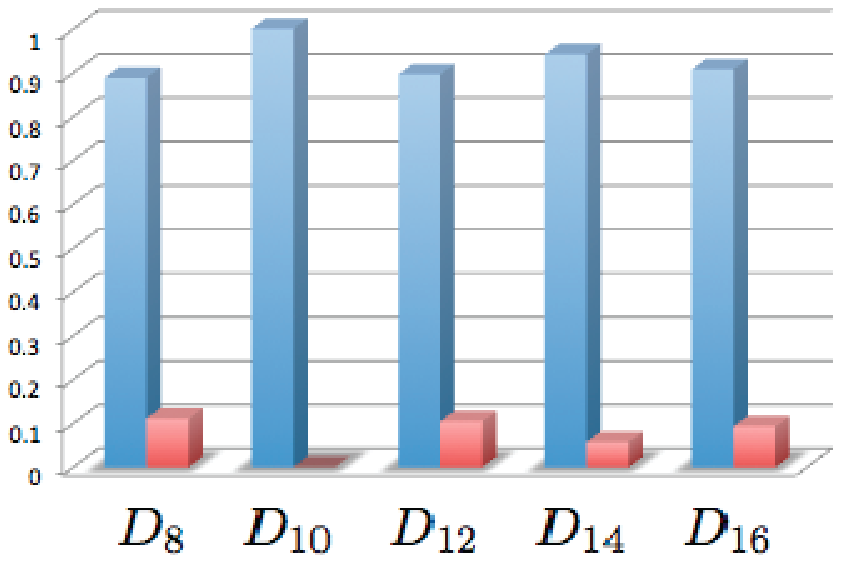} \includegraphics[width=5.5cm]{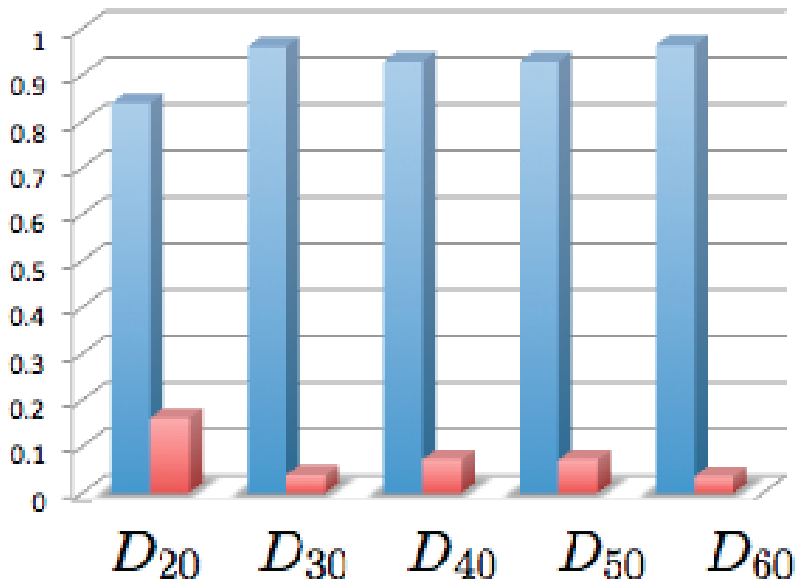}
\caption{\label{fig:charts} Relative number of sum-dominant sets (larger values) versus difference-dominant sets (lower values) in dihedral groups.}
\end{center}\end{figure}

As Figure \ref{fig:charts} suggests, sum-dominant sets are more likely to appear than difference-dominant sets. Let $S=R\cup F$ where $R$ is the set of rotations in $S$ and $F$ is the set of reflections in $S$. From Table \ref{table:setrotationflip} we note that the difference in what contributes to the sumsets and difference sets is $R-R$ which contributes to the difference set and $F-R$ and $R+R$ which contributes to the sumset. It is due to this that there are more sum-dominant sets than difference-dominant sets.

\begin{table}[h]
\begin{center}
  \begin{tabular}{| c | c | c  |}
    \hline
    \textbf{\texttt{Set}} & \textbf{\texttt{Rotations in the Set}} & \textbf{\texttt{Reflections in the Set}} \\ \hline
    S & $R$ & $F$ \\ \hline
    S+S & $R+R$, $F+F$ & $R+F$, $-R+F$ \\ \hline
    S-S & $R-R$, $F+F$ & $R+F$ \\
    \hline
  \end{tabular}\\
\caption{\label{table:setrotationflip} How elements contribute to the size of $S+S$ versus $S-S$.}
\end{center}
\end{table}


\section{Conclusion}\label{sec:conclusion}

We have shown that finite groups behave differently than the integers in the sense that almost all subsets are balanced. The reason is that finite groups do not have a fringe. As a result, in finite groups almost all sumsets and difference sets are equal to the entire group. The dihedral group case also hints at the importance of the size of the commutator subgroup and the number of order two elements. It is easy to see that the size of the sumset is greater when the commutator subgroup is small while the size of the difference set is lower due to the greater amount of order two elements.

A natural question to ask is what would happen if we no longer weight each subset equally. When each subset is chosen with uniform probability then the probability of the subset being balanced is equal to $1$; however, in $\mathbb{Z}/n\mathbb{Z}$, if we take subsets of the first half of the group (i.e., $\bar 0,\bar1, \dots,\bar \lfloor \frac{n}{2} \rfloor$) then the sumsets and difference sets behave like they would in $\mathbb{Z}$. Thus, the percentage of balanced groups is closer to $0$. It would be interesting to explore where the phase transition occurs.

Another question to ask is what happens when we look at non-abelian infinite groups. One difficulty is how we approach subsets of infinite groups. For example, if we look at $(\mathbb{Z}/2\mathbb{Z})^{|\mathbb{N}|}$ we have two different ways to limit the size of the subset. One possibility is to require $S$ to be a subset of a finite subgroup. This would allow for an easier computation of the limiting behavior, though we would have to determine the probability it lives in each finite subgroup.




\bigskip

\end{document}